\documentclass{amsart}

\usepackage{amsmath}
\usepackage{amsthm}
\usepackage{thmtools}	
\usepackage{amsfonts}
\usepackage{amssymb}
\usepackage{breqn}		
\usepackage{graphicx}   
\usepackage{lscape}		
\usepackage{verbatim}   
\usepackage{color}      
\usepackage{subfigure}  
\usepackage{marvosym}   

\usepackage{hyperref}   
\usepackage{nameref}    
\usepackage{cleveref}   

\usepackage{enumerate}  
\usepackage{appendix}   

\newcounter{LemmaCounter}
\newcounter{TheoremCounter}

\newtheorem{theorem}[TheoremCounter]{Theorem}

\newtheorem{corollary}[TheoremCounter]{Corollary}

\newtheorem{definition}[TheoremCounter]{Definition}

\begin{document}
\title{Abel Summation of Ramanujan-Fourier Series}
\author{John Washburn}
\address{N128W12795 Highland Road \\ Germantown, WI 53022}
\urladdr{http://www.WashburnResearch.org}
\email{math@WashburnResearch.org}
\date{\today}

\begin{abstract}

The main result of this paper is:

Given two arithmetic functions, $f(n)$ and $g(n)$, having point-wise convergent Ramanujan-Fourier (R-F) expansions of:
\begin{align*}
f(n) = \sum\limits_{q=1}^{\infty} a_q c_q(n)
&\quad \text{and} \quad
g(n) = \sum\limits^{\infty}_{q=1} b_q c_q(n)
\end{align*}
and there exist two finite bounds (not necessarily distinct), $K_f$ and $K_g$, 
such that for all $q \ge 1$ and all $n \ge 1$, 
$a_q$ and $b_q$, are bounded by:
\begin{align*}
\left\vert a_q \phi\left( q \right) \right\vert \le K_f
&\quad \text{and} \quad
\left\vert b_q \phi\left( q \right) \right\vert \le K_g
\end{align*}
where $\phi\left( q \right)$ is the Euler totient function.
\\

Then:
\begin{equation} \label{Eq:Abstract:IsAbelSummable}
\lim_{N \rightarrow \infty} \frac{1}{N} \sum^{N}_{n=1} f(n) \overline{g(n \pm m)}
\quad \text{is Abel summable to} \quad
\sum_{q=1}^{\infty} a_q \overline{b_q} c_q(m)
\end{equation}

When applied to the auto-correlation function of an arithmetic function, $f(n)$, the result is
the Wiener-Khichtine theorem as applied to Ramanujan-Fourier series.  Namely:
\begin{equation} \label{Eq:Abstract:WienerKhinchinForRamanujanSums}
\lim_{N \rightarrow \infty} \frac{1}{N} \sum^{N}_{n=1}
f\left( n \right)
\overline{f \left( n \pm m \right)}
=
\sum_{q=1}^{\infty} \Vert a_q \Vert^2 c_q(m)
\end{equation}

Equation \eqref{Eq:Abstract:WienerKhinchinForRamanujanSums} is the
the central relationship to be proved in within the works of
H. G. Gadiyar and R. Padma,
\cite{BibItem:GadiyarPadma:WienerKhinchin},
\cite{BibItem:GadiyarPadma:SophieGermainPrimes},
\cite{BibItem:GadiyarPadma:RotaMeetsRamanujan}, and
\cite{BibItem:GadiyarPadma:LinkingCircleAndSieve},
as related  to the following conjectures in number theory:
\begin{enumerate} 
\item The Hardy-Littlewood prime k-tuple conjecture for k=2.
\item The twinned prime conjecture.
\item The Sophie Germaine primes conjecture.
\end{enumerate}

While the statement in \eqref{Eq:Abstract:IsAbelSummable} is weaker than
\eqref{Eq:Abstract:WienerKhinchinForRamanujanSums},
it is powerful evidence supporting the conjecture by H. G. Gadiyar and R. Padma
that \eqref{Eq:Abstract:WienerKhinchinForRamanujanSums} is true.

\end{abstract}

\maketitle

\newpage
\section{Introduction}

One of the remarkable achievements of
Ramanujan\cite{BibItem:Ramanujan01}, Hardy\cite{BibItem:Hardy01} and Carmichael\cite{BibItem:Carmichael01}
was the development of Ramanujan-Fourier series which converge (point-wise) to an arithmetic function.
The Ramanujan-Fourier series, for an arithmetic function, $f(n)$, can given by
\begin{equation*}
f(n) = \sum\limits_{q=1}^{\infty} a_q c_q(n)
\end{equation*}
where the Ramanujan sum, $c_q(n)$, is defined as
\begin{equation*}
c_q(n) = \sum\limits^{q}_{\substack{k=1 \\ (k,q)=1}} e^{2 \pi i \frac{k}{q} n }
\end{equation*}
and $(k,q)$ is the greatest common divisor of $k$ and $q$.
The Ramanujan-Fourier coefficient, $a_q$, is given by:
\begin{equation*}
a_q = \frac{1}{\phi(q)} \lim_{N \to \infty} \sum_{n=1}^{N} f(n) c_q(n)
\end{equation*}
where $\phi(q)$ is the Euler totient function.
\\
Some of the properties of $c_q(n)$ and the mean value of $c_q(n)$ are:
\begin{enumerate}[ a) ] \label{List:Properties:RamanujanSums}
\item
\begin{equation*}
c_1(n) = 1 \quad \text{for all $n$.} \label{CqProperty:c1_Equals_1}
\end{equation*}
\item
\begin{equation*} \label{CqProperty:c0=Phi}
c_q(0) = \phi(q) \quad \text{for all $q > 1$.}
\end{equation*}
\item
\begin{equation*} \label{CqProperty:cn_Equals_DividesOrNot}
c_q(n) = \begin{cases}
\phi(q) & q \mid  n
\\
-1 & q \not\mid  n
\end{cases}
\end{equation*}
\item
\begin{equation*} \label{CqProperty:cn_Equals_AbsoluteValue}
\left\vert c_q(n) \right\vert \le \left\vert \phi(q) \right\vert \quad \forall \, n,q  \ge 1
\end{equation*}
\item
\begin{equation*} \label{CqProperty:EvenFunction}
c_q(n) = c_q(-n) \quad \text{for all $q$ and all $n$.}
\end{equation*}
\item
\begin{equation*} \label{CqProperty:MeanValue:Single}
\lim_{N \to \infty} \sum_{n=1}^{N} c_q(n) =
\begin{cases}
1 & q=1 \\
0 & \text{otherwise} \\
\end{cases}
\end{equation*}
\item
\begin{equation*} \label{CqProperty:MeanValue:Convolution}
\lim_{N \to \infty} \sum_{n=1}^{N} c_{q_1}(n) c_{q_2}(n \pm m) =
\begin{cases}
c_{q_1}(m) & {q_1} = {q_2} \\
0 & {q_1} \ne {q_2} \\
\end{cases}
\end{equation*}
\end{enumerate}

\newpage
\section{Definitions and Notation}

This paper has the following notational conventions:
\begin{itemize}
\item $n$ is a positive integer.
\item $x$ is real.
\item $p$ is an arbitrary prime.
\item $q$ is a positive integer unless the context clearly indicates otherwise.
\item $z$ is a real within the open interval: $0 < z < 1$.
\end{itemize}

\begin{definition} \label{def:RamanujanSum:Integer}
For all integers $n \in \mathbb{Z}$,
the classical (integer) Ramanujan sum, $c_q\left( n \right)$, is defined as:
\begin{equation*}
c_q\left( n \right) =
\sum\limits_{\substack{k=1 \\ (k,q)=1}}^{q}
e^{ 2 \pi \frac{k}{q} n }
\end{equation*}
or equivalently as:
\begin{equation*}
c_q\left( n \right) =
\begin{cases}
1 & q = 1
\\
\left( -1 \right)^n & q = 2
\\
2 \sum\limits_{\substack{k=1 \\ (k,q)=1}}^{\left\lfloor \frac{q}{2} \right\rfloor}
\cos\left( 2 \pi \frac{k}{q} n \right) & q \ge 3
\end{cases}
\end{equation*}
\end{definition}

\begin{definition} \label{def:RamanujanSum:Real}
The Ramanujan sums can be extended to the reals by defining, $c_q\left( x \right)$, as:
\begin{equation*}
c_q\left( x \right) =
\begin{cases}
1 & q = 0
\\
\cos\left( 2 \pi x \right) & q = 1
\\
\cos\left( \pi x \right) & q = 2
\\
2 \sum\limits_{\substack{k=1 \\ (k,q)=1}}^{\left\lfloor \frac{q}{2} \right\rfloor}
\cos\left( 2 \pi \frac{k}{q} x \right) & q \ge 3
\end{cases}
\end{equation*}
\end{definition}

\begin{corollary}
A corollary of \Cref{def:RamanujanSum:Real} is
\begin{equation*}
c_q(x) = c_q(n)
\end{equation*}
for all $q \ge 1$ and $x \in \mathbb{Z}$
\end{corollary}

\begin{definition} \label{def:RamanujanExpansion:PowerSeries}
Given an arithmetic function, $f(n)$, with a point-wise convergent Ramanujan-Fourier (R-F) expansion of:
\begin{equation*}
f(n) = \sum\limits_{q=1}^{\infty} a_q c_q(n)
\end{equation*}
then the related functions, $f(z, n)$ and $f\left(z, x\right)$,
are the Ramanujan expansion power series and are given by:
\begin{equation*}
\begin{aligned}
f(z, n)
&\stackrel{def}{=}
\sum\limits_{q=1}^{\infty} a_q z^q c_q(n)
\\
f\left(z, x\right)
&\stackrel{def}{=}
\sum\limits_{q=1}^{\infty} a_q z^q c_q(x)
\end{aligned}
\end{equation*}
\end{definition}

\begin{definition} \label{def:RamanujanExpansion:PartialSummation}
Given a Ramanujan expansion power series of:
\begin{align*}
f(z, n) &= \sum\limits_{q=1}^{\infty} a_q z^q c_q(n)
\\
&\text{or}
\\
f\left(z, x\right) &= \sum\limits_{q=1}^{\infty} a_q z^q c_q(x)
\end{align*}
then the partial summation of the Ramanujan expansion power series for the first $Q$ terms is given by:
\begin{align*}
f_Q\left(z, n \right) 
&\stackrel{def}{=}
\sum\limits_{q=1}^{Q} a_q z^q c_q(n)
\\
f_Q\left(z, x \right) 
&\stackrel{def}{=}
\sum\limits_{q=1}^{Q} a_q z^q c_q(x)
\end{align*}
\end{definition}

\newpage

\begin{theorem}[Uniform Convergence] \label{lemma:UniformConvergenceConvergence:FQzx}
Given an Ramanujan expansion series, $f\left(z, x\right)$, with an R-F expansion of
\begin{equation*}
f\left(z, x\right) = \sum\limits_{q=1}^{\infty} a_q z^q c_q(x)
\end{equation*}
where, 
\begin{itemize}
\item for all $q \ge 1$, and
\item for all $x \in \mathbb{R}$, and
\item for some finite bound, $K_f$, and
\item $\phi(q)$ is the Euler totient function
\end{itemize}
the R-F coefficients, $a_q$, are bounded by:
\begin{equation*}
\left\vert a_q \phi(q) \right\vert \le K_f
\end{equation*}
\\

Then, the sequence of partial summations, $f_{Q}(z, x)$, converges uniformly in $x$ to $f\left(z, x\right)$
as $Q \to \infty$.
\end{theorem}

\begin{proof}
\begin{equation}
f\left(z, x\right) = \sum\limits_{q=1}^{\infty} a_q z^q c_q(x)
\end{equation}

\begin{equation}
f\left(z, x\right) = \sum\limits_{q=1}^{Q} a_q z^q c_q(x)  + \sum\limits_{q=Q+1}^{\infty} a_q z^q c_q(x)
\end{equation}

\begin{equation}
f\left(z, x\right) - \sum\limits_{q=1}^{Q} a_q z^q c_q(x)  = \sum\limits_{q=Q+1}^{\infty} a_q z^q c_q(x)
\end{equation}

\begin{equation}
\left\vert f\left(z, x\right) - f_Q\left(z, x\right)  \right\vert \le \left\vert \sum\limits_{q=Q+1}^{\infty} a_q z^q c_q(x) \right\vert
\end{equation}

\begin{equation} \label{eq:AfterTriangleInequality}
\left\vert f\left(z, x\right) - f_Q\left(z, x\right)  \right\vert \le \sum\limits_{q=Q+1}^{\infty} \left\vert a_q c_q(x) \right\vert \left\vert z^q \right\vert
\end{equation}
Since $\left\vert a_q c_q(x) \right\vert \le K_f$ and $0 < z$ \eqref{eq:AfterTriangleInequality} becomes:
\begin{equation}
\left\vert f\left(z, x\right) - f_Q\left(z, x\right)  \right\vert \le M_f \sum\limits_{q=Q+1}^{\infty} z^q
\end{equation}

\begin{equation} 
\left\vert f\left(z, x\right) - f_Q\left(z, x\right)  \right\vert \le M_f \left( \frac{ z^{Q+1}}{1-z} \right)
\end{equation}

\begin{equation} \label{eq:PartialSummation:Tail:Bounds}
\left\vert f\left(z, x\right) - f_Q\left(z, x\right)  \right\vert \le z^{Q} \left( \frac{ z M_f }{1-z} \right)
\end{equation}

Since, $0 < z < 1$, the right hand side of the \eqref{eq:PartialSummation:Tail:Bounds}
forms a strictly decreasing sequence in $Q$ such that:
\begin{equation*}
z^{Q+N}
\left( \frac{ z  M_f  }{1-z} \right)
<
z^{Q+N-1}
\left( \frac{ z  M_f  }{1-z} \right)
\cdots
<
z^{Q+2}
\left( \frac{ z  M_f  }{1-z} \right)
<
z^{Q+1}
\left( \frac{ z  M_f  }{1-z} \right)
<
z^Q
\left( \frac{ z  M_f  }{1-z} \right)
\end{equation*}
for all integers $N \ge 0$.

Thus, for any fixed $z$ in the interval: $0 < z < 1$, and for any $0 < \epsilon$, it is possible to select a value, $Q_z\left( \epsilon \right)$,
such that
\begin{equation*}
z^{Q} \left( M_f  \frac{ z }{1-z} \right) \le \epsilon
\end{equation*}
for all $Q \ge Q_z\left( \epsilon \right)$.

Since the selection of $Q_z\left( \epsilon \right)$ depends only on $z$ and does not depend on $x$,
the sequence of functions, $f_Q\left(z, x \right)$,
converges uniformly to $f\left(z, x \right)$ in $x \in \mathbb{R}$ 
to the Ramanujan expansion power series, $f\left(z, x \right)$
for every fixed $z$ on the interval $0 < z < 1$.
\end{proof}

\begin{corollary} \label{cor:UniformConvergenceConvergence:FQzx:Integers}
A corollary to \Cref{lemma:UniformConvergenceConvergence:FQzx} is 
for any fixed $z$ within the open interval, $0 < z < 1$, the sequence of functions,
$f_Q\left(z, x \right)$,
converges uniformly to $f\left(z, n \right)$
for all $n \in \mathbb{Z}$
\end{corollary}

\newpage
\section{Proof of \Cref{thm:AbelSummationTheorem}}

\begin{theorem}[Abel Summability] \label{thm:AbelSummationTheorem}
Given two arithmetic functions, $f(n)$ and $g(n)$, with point-wise convergent Ramanujan-Fourier expansions of:
\begin{equation*}
f(n) = \sum\limits_{q=1}^{\infty} a_q c_q(n) \qquad
\text{and} \qquad
g(n) = \sum\limits^{\infty}_{q=1} b_q c_q(n)
\end{equation*}
where, 
\begin{itemize}
\item for all $q \ge 1$, and
\item for all $x \in \mathbb{R}$,  
\item for some finite bound, $K_f$, and
\item $\phi(q)$ is the Euler totient function
\end{itemize}
the R-F coefficients, $a_q$ and $b_q$ , are bounded by:
\begin{equation*}
\left\vert a_q \phi(q) \right\vert \le K_f
\text{ and }
\left\vert b_q \phi(q) \right\vert \le K_g
\end{equation*}
\\

Then
\begin{equation} \label{eq:BasicConvolutionExpression}
\lim_{N \rightarrow \infty} \frac{1}{N} \sum^{N}_{n=1} f(n) \overline{g(n \pm m)}
\end{equation}
is Abel summable to:
\begin{equation*}
\sum_{q=1}^{\infty} a_q \overline{b_q} c_q(m)
\end{equation*}
\end{theorem}

\begin{proof}

For any fixed $z$ in the open interval $0 < z < 1$ the two Ramanujan expansion power series of
$f(n)$ and $g(n)$ are given by:
\begin{equation} \label{eq:DefinitionParameterizedArithmeticFunctions}
f(z, n) = \sum\limits_{q=1}^{\infty} a_q z^q c_q(n) \qquad
\text{and} \qquad
g(z, n) = \sum\limits^{\infty}_{q=1} b_q z^q c_q(n)
\end{equation}
By hypothesis, both $f(n)$ and $g(n)$ converge, thus, by Abel's Theorem, 
as the value of $z$ pushes up to 1 the Ramanujan expansion power series of \eqref{eq:DefinitionParameterizedArithmeticFunctions}
converge to:
\begin{equation}
\lim_{z \to 1^{-}} f(z, n) = f(n)\qquad
\text{and} \qquad
\lim_{z \to 1^{-}} g(z, n) = g(n)\qquad
\end{equation}

Since, by hypothesis, the Ramanujan expansion power series meet the requirements of 
\Cref{cor:UniformConvergenceConvergence:FQzx:Integers}
there exist two constants,  $Q_1$ and $Q_2$, such that:
\begin{equation}
\left\vert f(z, n) - f_{Q_1}(z, n) \right\vert < \epsilon \qquad
\text{and} \qquad
\left\vert g(z, n \pm m) - g_{Q_2}(z, n \pm m) \right\vert < \epsilon \qquad
\end{equation}

Selecting $Q_{\epsilon}$ as $\text{max}\left( Q_1, Q_2 \right)$ yields:
\begin{equation} \label{eq:CoreInequality:UniformConvergence}
\left\vert f(z, n) - f_{Q_{\epsilon}}(z, n) \right\vert < \epsilon \qquad
\text{and} \qquad
\left\vert g(z, n \pm m) - g_{Q{\epsilon}}(z, n \pm m) \right\vert < \epsilon \qquad
\end{equation}

The 2 inequalities in \eqref{eq:CoreInequality:UniformConvergence} can be stated as:
\begin{subequations} \label{eq:CoreInequality:AsBoundedInterval}
\begin{align}
f_{Q_{\epsilon}}(z, n) - \epsilon
\le
&f(z, n)
\le
f_{Q_{\epsilon}}(z, n) + \epsilon 
\label{eq:CoreInequality:AsBoundedInterval:Fzn}
\\
g_{Q_{\epsilon}}(z, n \pm m) - \epsilon
\le
&g(z, n \pm m)
\le
g_{Q_{\epsilon}}(z, n \pm m) + \epsilon
\label{eq:CoreInequality:AsBoundedInterval:Gzn}
\end{align}
\end{subequations}

Beginning with the R-F expansions of $f(z, n)$ and $g(z, n \pm m)$,
the convolution of $f(z,n)$ and $g(z,n \pm m)$ is given by:
\begin{equation} \label{eq:InitialSummation}
\lim_{N \to \infty} \frac{1}{N}  \sum_{n=1}^{N}
f(z, n) \overline{g(z, n \pm m)} =
\lim_{N \to \infty} \frac{1}{N}  \sum_{n=1}^{N}
\left[
\sum\limits_{q_1=1}^{\infty} a_{q_1} z^{q_1} c_{q_1}(n)
\sum\limits_{q_2=1}^{\infty} \overline{b_{q_2}} z^{q_2} c_{q_2}(n \pm m)
\right]
\end{equation}

Applying the right hand sides of the inequalities 
\eqref{eq:CoreInequality:AsBoundedInterval:Fzn} and \eqref{eq:CoreInequality:AsBoundedInterval:Gzn} 
to \eqref{eq:InitialSummation} yields:

\begin{equation}
\lim_{N \to \infty} \frac{1}{N}  \sum_{n=1}^{N}
f(z, n) \overline{g(z, n \pm m)} \le
\lim_{N \to \infty} \frac{1}{N}  \sum_{n=1}^{N}
\left[
\epsilon +
\sum\limits_{q_1=1}^{Q_{\epsilon}} a_{q_1} z^{q_1} c_{q_1}(n)
\right]
\left[
\epsilon +
\sum\limits_{q_2=1}^{Q_{\epsilon}} \overline{b_{q_2}} z^{q_2} c_{q_2}(n \pm m)
\right]
\end{equation}

\begin{equation}
\begin{aligned}
\lim_{N \to \infty} \frac{1}{N}  \sum_{n=1}^{N}
f(z, n) \overline{g(z, n \pm m)}
&\le
\lim_{N \to \infty} \frac{1}{N}  \sum_{n=1}^{N}
\left[
\sum\limits_{q_1=1}^{Q_{\epsilon}} a_{q_1} z^{q_1} c_{q_1}(n)
\sum\limits_{q_2=1}^{Q_{\epsilon}} \overline{b_{q_2}} z^{q_2} c_{q_2}(n \pm m)
\right]
\\
&+
\lim_{N \to \infty} \frac{1}{N}  \sum_{n=1}^{N}
\left[
\epsilon
\sum\limits_{q_1=1}^{Q_{\epsilon}} a_{q_1} z^{q_1} c_{q_1}(n)
\right]
\\
&+
\lim_{N \to \infty} \frac{1}{N}  \sum_{n=1}^{N}
\left[
\epsilon
\sum\limits_{q_2=1}^{Q_{\epsilon}} \overline{b_{q_2}} z^{q_2} c_{q_2}(n \pm m)
\right]
\\
&+
\lim_{N \to \infty} \frac{1}{N}  \sum_{n=1}^{N}
\left[
{\epsilon}^2
\right]
\end{aligned}
\end{equation}

\begin{equation} \label{eq:PenultimateAfterLimitsExchanged}
\begin{aligned}
\lim_{N \to \infty} \frac{1}{N}  \sum_{n=1}^{N}
f(z, n) \overline{g(z, n \pm m)}
&\le
\sum\limits_{q_1=1}^{Q_{\epsilon}}
\sum\limits_{q_2=1}^{Q_{\epsilon}}
a_{q_1} \overline{b_{q_2}}
z^{q_1} z^{q_2}
\left[
\lim_{N \to \infty} \frac{1}{N}  \sum_{n=1}^{N}
c_{q_1}(n)
c_{q_2}(n \pm m)
\right]
\\
&+
\epsilon
\sum\limits_{q_1=1}^{Q_{\epsilon}} a_{q_1} z^{q_1}
\left[
\lim_{N \to \infty} \frac{1}{N}  \sum_{n=1}^{N}
c_{q_1}(n)
\right]
\\
&+
\epsilon
\sum\limits_{q_2=1}^{Q_{\epsilon}} \overline{b_{q_2}} z^{q_2}
\left[
\lim_{N \to \infty} \frac{1}{N}  \sum_{n=1}^{N}
c_{q_2}(n \pm m)
\right]
\\
&+
{\epsilon}^2
\left[
\lim_{N \to \infty} \frac{1}{N}  \sum_{n=1}^{N}
1
\right]
\end{aligned}
\end{equation}

Applying the limit process, $N \to \infty$ and
using the properties:
\ref{CqProperty:cn_Equals_AbsoluteValue},
\ref{CqProperty:MeanValue:Single}, and
\ref{CqProperty:MeanValue:Convolution}
of $c_q(n)$ listed in the introduction,
equation \eqref{eq:PenultimateAfterLimitsExchanged} becomes:
\begin{equation}
\begin{aligned}
\lim_{N \to \infty} \frac{1}{N}  \sum_{n=1}^{N}
f(z, n) \overline{g(z, n \pm m)}
&\le
\sum\limits_{q=1}^{Q_{\epsilon}}
a_{q} \overline{b_{q}}
z^{2 q}
c_{q}(m)
\\
&+
\epsilon
\left[
a_{1} z^{1}
\right]
\\
&+
\epsilon
\left[
\overline{b_{1}} z^{1}
\right]
\\
&+
{\epsilon}^2
\end{aligned}
\end{equation}

\begin{equation}
\begin{aligned}
\lim_{N \to \infty} \frac{1}{N}  \sum_{n=1}^{N}
f(z, n) \overline{g(z, n \pm m)}
&\le
\sum\limits_{q=1}^{Q_{\epsilon}}
a_{q} \overline{b_{q}}
z^{2 q}
c_{q}(m)
\\
&+
\epsilon \left\vert a_{1} \right\vert
\\
&+
\epsilon \left\vert \overline{b_{1}} \right\vert
\\
&+
{\epsilon}^2
\end{aligned}
\end{equation}

Thus, the upper bound is given by:
\begin{equation} \label{eq:SandwichInequality:UpperLimit}
\lim_{N \to \infty} \frac{1}{N}  \sum_{n=1}^{N}
f(z, n) \overline{g(z, n \pm m)}
\le
\left(
\sum\limits_{q=1}^{Q_{\epsilon}}
a_{q} \overline{b_{q}}
z^{2 q}
c_{q}(m)
\right)
+
\epsilon \left( \left\vert a_{1} \right\vert + \left\vert \overline{b_{1}} \right\vert \right)
+
{\epsilon}^2
\end{equation}

Applying the left hand sides of the inequalities from
\eqref{eq:CoreInequality:AsBoundedInterval:Fzn} and \eqref{eq:CoreInequality:AsBoundedInterval:Gzn} 
to \eqref{eq:InitialSummation}
and using similar reasoning,
it can be shown that the lower bound is given by:
\begin{equation} \label{eq:SandwichInequality:LowerLimit}
\lim_{N \to \infty} \frac{1}{N}  \sum_{n=1}^{N}
f(z, n) \overline{g(z, n \pm m)}
\ge
\left(
\sum\limits_{q=1}^{Q_{\epsilon}}
a_{q} \overline{b_{q}}
z^{2 q}
c_{q}(m)
\right)
-
\epsilon \left( \left\vert a_{1} \right\vert + \left\vert \overline{b_{1}} \right\vert \right)
+
{\epsilon}^2
\end{equation}

Combining the upper and lower bounds of
\eqref{eq:SandwichInequality:UpperLimit} and
\eqref{eq:SandwichInequality:LowerLimit}
yields:
\begin{equation}
\begin{aligned}
\left(
\sum\limits_{q=1}^{Q_{\epsilon}}
a_{q} \overline{b_{q}}
z^{2 q}
c_{q}(m)
\right)
-
\epsilon \left( \left\vert a_{1} \right\vert + \left\vert \overline{b_{1}} \right\vert \right)
+
{\epsilon}^2
&\le
\lim_{N \to \infty} \frac{1}{N}  \sum_{n=1}^{N}
f(z, n) \overline{g(z, n \pm m)}
\\
&\le
\left(
\sum\limits_{q=1}^{Q_{\epsilon}}
a_{q} \overline{b_{q}}
z^{2 q}
c_{q}(m)
\right)
+
\epsilon \left( \left\vert a_{1} \right\vert + \left\vert \overline{b_{1}} \right\vert \right)
+
{\epsilon}^2
\end{aligned}
\end{equation}

\begin{equation} \label{eq:LeftHandSideIsNotZero}
\begin{aligned}
{\epsilon}^2
-
\epsilon \left( \left\vert a_{1} \right\vert + \left\vert \overline{b_{1}} \right\vert \right)
&\le
\lim_{N \to \infty} \frac{1}{N}  \sum_{n=1}^{N}
f(z, n) \overline{g(z, n \pm m)}
-
\left(
\sum\limits_{q=1}^{Q_{\epsilon}}
a_{q} \overline{b_{q}}
z^{2 q}
c_{q}(m)
\right)
\\
&\le
{\epsilon}^2
+
\epsilon \left( \left\vert a_{1} \right\vert + \left\vert \overline{b_{1}} \right\vert \right)
\end{aligned}
\end{equation}

Due to the point-wise convergence of the R-F expansions for
$f(n)$ and $g(n)$,
both $a_1$ and $b_1$ are necessarily finite.
The value of $\epsilon$ can be choose to arbitrarily small.
Thus, by selecting  a value of $\epsilon$ such that
\begin{equation*}
\epsilon =
\begin{cases}
0 < \epsilon \le 1 & a_{1} = b_{1} = 0
\\
0 < \epsilon \le
\left( \left\vert a_{1} \right\vert + \left\vert \overline{b_{1}} \right\vert \right)
& otherwise
\end{cases}
\end{equation*}
equation \eqref{eq:LeftHandSideIsNotZero} becomes:

\begin{equation}  \label{eq:SandwichInequality:Epsilon}
0 \le
\left\vert
\lim_{N \to \infty} \frac{1}{N}  \sum_{n=1}^{N}
f(z, n) \overline{g(z, n \pm m)}
-
\left(
\sum\limits_{q=1}^{Q_{\epsilon}}
a_{q} \overline{b_{q}}
z^{2 q}
c_{q}(m)
\right)
\right\vert
\\
\le
{\epsilon}^2
+
\epsilon \left( \left\vert a_{1} \right\vert + \left\vert \overline{b_{1}} \right\vert \right)
\end{equation}

Applying the squeeze theorem to \eqref{eq:SandwichInequality:Epsilon}
yields:
\begin{equation}
\lim_{\substack{Q_{\epsilon} \to \infty \\ \epsilon \to 0}}
\left\vert
\lim_{N \to \infty} \frac{1}{N}  \sum_{n=1}^{N}
f(z, n) \overline{g(z, n \pm m)}
-
\sum\limits_{q=1}^{Q_{\epsilon}}
a_{q}  \overline{b_{q}} z^{2 q} c_{q}(m)
\right\vert = 0
\end{equation}
which can be restated as:
\begin{equation} \label{eq:Result:AbelSummation:Step1}
\lim_{N \to \infty} \frac{1}{N}  \sum_{n=1}^{N}
f(z, n) \overline{g(z, n \pm m)}
=
\sum\limits_{q=1}^{\infty}
a_{q}  \overline{b_{q}} z^{2 q} c_{q}(m)
\end{equation}

Letting $z$ push up to 1 in equation \eqref{eq:Result:AbelSummation:Step1} yields:
\begin{equation} \label{eq:Result:AbelSummation:Step2}
\begin{aligned}
\lim_{z \to 1^{-}}
\lim_{N \to \infty} \frac{1}{N}  \sum_{n=1}^{N}
f(z, n) \overline{g(z, n \pm m)}
&=
\lim_{z \to 1^{-}}
\sum\limits_{{q}=1}^{\infty}
a_{q}  \overline{b_{q}} z^{2 q} c_{q}(m)
\\
\lim_{N \to \infty} \frac{1}{N}  \sum_{n=1}^{N}
f(n) \overline{g(n \pm m)}
&\to
\sum\limits_{{q}=1}^{\infty}
a_{q}  \overline{b_{q}} c_{q}(m)
\end{aligned}
\end{equation}

Thus:
\begin{equation*}
\lim_{N \to \infty} \frac{1}{N}  \sum_{n=1}^{N}
f(n) \overline{g(n \pm m)}
\end{equation*}
is Abel summable to:
\begin{equation*}
\sum\limits_{{q}=1}^{\infty}
a_{q}  \overline{b_{q}} c_{q}(m)
\end{equation*}

\end{proof}

\newpage
\section{Applicability of \Cref{thm:AbelSummationTheorem} to the work of Gadiyar and Padma}

The arithmetic function used by Gadiyar and Padma
in both \cite{BibItem:GadiyarPadma:WienerKhinchin} and \cite{BibItem:GadiyarPadma:SophieGermainPrimes}
is:
\begin{equation}
\frac{\phi(n)}{n} \Lambda\left( n \right) = \sum\limits_{q=1}^{\infty} \frac{\mu(q)}{\phi(q)} c_q(n)
\end{equation}

From this $a_q = \frac{\mu(q)}{\phi(q)}$ and
\begin{align*}
\left\vert a_q c_q(x) \right\vert
&=
\left\vert \frac{\mu(q)}{\phi(q)} c_q(x) \right\vert 
\\
&\le
\left\vert \frac{\mu(q)}{\phi(q)}  \right\vert \left\vert c_q(x) \right\vert 
\\
&\le
\frac{1}{\phi(q)} \phi(q)
\\
&\le 1
\end{align*}
Thus, \Cref{thm:AbelSummationTheorem} applies to the auto-correlation of the arithmetic function:
\begin{equation}
\frac{\phi(n)}{n} \Lambda\left( n \right)
\end{equation}

For paper \cite{BibItem:GadiyarPadma:WienerKhinchin} on the prime pair conjecture and, a fortiori, on the twin prime conjecture,
\begin{equation}
\lim_{N \to \infty} \sum\limits_{n=1}^{N}
\left( \frac{\phi(n) \Lambda\left( n \right)}{n}  \right)
\left( \frac{\phi(n + h) \Lambda\left( n + h \right)}{n + h}  \right)
\end{equation}
is Abel summable to $C\left( h \right)$ which is defined in \cite{BibItem:GadiyarPadma:WienerKhinchin} by equation 7 as:
\begin{equation}
\begin{aligned}
\sum\limits_{q=1}^{\infty}
\frac{{\mu}^2(q)}{{\phi}^2(q)}
c_q\left( h \right)
&=
C\left( h \right)
\\
&= \begin{cases}
2
\prod\limits_{\substack{p \mid h \\ p > 2}} \left( \frac{p-1}{p-2} \right)
\prod\limits_{p>2} \left( 1 - \frac{1}{\left( p - 1 \right)^2} \right)
& \text{h is even}
\\
0 & \text{h is odd}
\end{cases}
\end{aligned}
\end{equation}

Similarly, for paper \cite{BibItem:GadiyarPadma:SophieGermainPrimes} on Sophie German primes,
\begin{equation}
\lim_{N \to \infty}
\frac{
\Phi_{\left( a,b,\ell \right)}\left( N \right)
}{N}
\end{equation}
is Abel-summable to the expression $S$ defined in equation 16 of \cite{BibItem:GadiyarPadma:SophieGermainPrimes}
\begin{equation}
S = \begin{cases}
\frac{2 C}{a}
\prod\limits_{\substack{p > 2 \\ p \vert \left( a b \ell \right)}}
\left( \frac{p-1}{p-2} \right)
&\text{if} \, (a,b) = (a,\ell) = (b,\ell) = 1 \, \text{exactly one of $a$, $b$, $\ell$ is even}
\\
0 &\text{otherwise}
\end{cases}
\end{equation}

\newpage
\section{Conclusions}
The problem with proving the strong form of the Wiener-Khinchin theorem,
equation \eqref{Eq:Abstract:WienerKhinchinForRamanujanSums}, is two fold.
\\

The first is the question:\\
\em
Does the limit
\begin{equation} \label{eq:WinerKhinchinLimit}
\lim_{N \rightarrow \infty} \frac{1}{N} \sum^{N}_{n=1}
f\left( n \right)
\overline{f \left( n \pm m \right)}
\end{equation}
exist or not?
\em
\\

The second is the question:\\
\em
If the limit in \eqref{eq:WinerKhinchinLimit} exists, then what is the value of that limit?
\em
\\

\Cref{thm:AbelSummationTheorem} answers the second question, but not the first.
If the limit in \eqref{eq:WinerKhinchinLimit} exists, then the value of that limit is
\begin{equation*}
\sum_{q=1}^{\infty} \left\Vert a_q \right\Vert^2 c_q(m)
\end{equation*}

With the value of the limit established, the task at hand is for the author to
prove
\begin{equation} \label{eq:Limit2BFound}
\lim_{N \to \infty} \frac{1}{N}
\sum_{n=1}^{N}
\left[ \frac{\phi\left( n     \right) \Lambda\left( n     \right)}{n  } \right]
\left[ \frac{\phi\left( n + h \right) \Lambda\left( n + h \right)}{n+h} \right]
\end{equation}
converges to some finite limit.
If a finite limit to \eqref{eq:Limit2BFound} exists, then the results developed in
\cite{BibItem:GadiyarPadma:WienerKhinchin},
\cite{BibItem:GadiyarPadma:SophieGermainPrimes},
\cite{BibItem:GadiyarPadma:RotaMeetsRamanujan}, and
\cite{BibItem:GadiyarPadma:LinkingCircleAndSieve},
are proved.

\end{document}